\RequirePackage{amsmath}
\documentclass[runningheads]{llncs}
\usepackage{amssymb,amsfonts}
\usepackage{cite}
\usepackage{graphicx}
\usepackage{array}
\usepackage{algorithm}
\usepackage{algorithmicx}
\usepackage[noend]{algpseudocode}
\usepackage[verbose]{wrapfig}
\usepackage[colorlinks,linkcolor=cyan,citecolor=blue]{hyperref}
\usepackage{multirow}
\usepackage{enumerate}

\usepackage{array}
\usepackage{booktabs}   
\usepackage{diagbox}    
\usepackage[x11names]{xcolor}      
\usepackage{color}
\usepackage{colortbl}
\usepackage{adjustbox}

\usepackage{pgfplots}
\usepackage{tikz}

\pgfplotsset{compat=1.17} 

\renewcommand{\arraystretch}{1.3}

\newcommand{\lvec}[1]{\overset{{}_{\leftarrow}}{#1}}

\renewcommand{\S}{\mathcal{S}}

\newcommand{\T}{\mathcal{T}}
\newcommand{\ST}{\mathcal{S\!T}}
\newcommand{\R}{\mathcal{DR}}
\newcommand{\F}{\mathcal{DFR}}
\newcommand{\perm}{\mathsf{Perm}}
\newcommand{\rep}{\mathsf{Rep}}

\newcommand{\prim}{\mathsf{Prim}}
\newcommand{\dev}{\mathsf{Dev}}
\newcommand{\kmers}{\mathsf{kmers}}

\begin{document}

\title{Expected Density of Random Minimizers}

\author{Shay Golan\inst{1,2} \and Arseny M. Shur\inst{3}}

\institute{Reichman University, Herzliya, Israel 
\and  University of Haifa, Israel \\\email{golansh1@biu.ac.il} \and
Bar Ilan University, Ramat Gan, Israel\\ \email{shur@datalab.cs.biu.ac.il} }

\maketitle

\vspace*{-4mm}
\begin{abstract}
Minimizer schemes, or just minimizers, are a very important computational primitive in sampling and sketching biological strings.
Assuming a fixed alphabet of size $\sigma$, a minimizer is defined by two integers $k,w\ge2$ and a total order $\rho$ on strings of length $k$ (also called $k$-mers).
A string is processed by a sliding window algorithm that chooses, in each window of length $w+k-1$, its minimal $k$-mer with respect to $\rho$.
A key characteristic of the minimizer is the expected density of chosen $k$-mers among all $k$-mers in a random infinite $\sigma$-ary string.
Random minimizers, in which the order $\rho$ is chosen uniformly at random, are often used in applications.
However, little is known about their expected density $\R_\sigma(k,w)$ besides the fact that it is close to $\frac{2}{w+1}$ unless $w\gg k$.

We first show that $\R_\sigma(k,w)$ can be computed in $O(k\sigma^{k+w})$ time.
Then we attend to the case $w\le k$ and present a formula that allows one to compute $\R_\sigma(k,w)$ in just $O(w\log w)$ time.
Further, we describe the behaviour of $\R_\sigma(k,w)$ in this case, establishing the connection between $\R_\sigma(k,w)$, $\R_\sigma(k+1,w)$, and $\R_\sigma(k,w+1)$.
In particular, we show that $\R_\sigma(k,w)<\frac{2}{w+1}$ (by a tiny margin) unless $w$ is small.
We conclude with some partial results and conjectures for the case $w>k$.
\end{abstract}

\section{Introduction}

The study of length-$k$ substrings (\emph{$k$-mers}) of long strings dates back to the conjectures of Golomb \cite{Gol67} and Lempel \cite{Lem71}, proved by Mykkeltveit in 1972 \cite{Mykk72}.
He constructed, for each $\sigma\ge 2$ and $k\ge 2$, a minimum-size \emph{unavoidable} set of $\sigma$-ary $k$-mers; ``unavoidable'' means that every long enough $\sigma$-ary string contains a substring from Mykkeltveit's set.
The interest to $k$-mers boosted in 2000s, when sketching of biological sequences was proposed as an alternative approach to full-text indexing with FM-index \cite{FeMa00} and similar tools.
Minimizers \cite{SWA03, RHHMY04} provide a very simple way to sample $k$-mers  for sketching and are also used in more involved sampling schemes such as syncmers \cite{Edg21}, strobemers \cite{Sah21}, and mod-minimizers \cite{KP24}.
For more information, see \cite{ZKM20,SBBM23,CHM21} and the references therein.

To sample substrings for a sketch of the string $S$, one fixes integer parameters $k$ and $w$ and chooses one $k$-mer in every set of $w$ consecutive $k$-mers in $S$; this process can be viewed as choosing a $k$-mer in each ``window'' of length $(w+k-1)$.
By \emph{markup} of $S$ we mean any map assigning to each window its chosen $k$-mer; we refer to the starting positions of these $k$-mers in $S$ as \emph{marked positions}.
A \emph{(sampling) scheme} is a deterministic algorithm taking a string $S$ and the numbers $k,w$ as the input and computing the markup of $S$.
A scheme must be \emph{local}, which means that the choice of a $k$-mer in a window depends solely on the window as a string.
A \emph{minimizer} is a scheme that fixes a linear order on $k$-mers and chooses the starting position of the minimal $k$-mer in each window, breaking ties to the left (breaking ties to the right instead, one will get ``dual'' minimizers having the same properties as minimizers).

The \emph{density} of a markup is the ratio between the number of marked positions and the length of $S$. 
For a scheme $\S$, let $D_{\S}(n)$ be the expected density of the markup of a uniformly random string of length $n$.
The \emph{density of $\S$} is defined as the limit $D_{\S}=\lim_{n\to\infty}D_{\S}(n)$.
Trivially, $D_\S\ge 1/w$ by the definition of markup.

Below we consider only minimizers.
Given $\sigma$ and $k$, $\S_\rho$ denotes the minimizer that uses the order $\rho$ on $k$-mers (one can view $\rho$ as a permutation of the set of all $\sigma$-ary $k$-mers).
To simplify the notation, we write $D_\rho$ instead of $D_{\S_\rho}$.
If $\rho$ is chosen uniformly at random from the set of all permutations of $\sigma$-ary $k$-mers, $D_\rho$ becomes a random variable.
Its expectation is called the \emph{density of the random order}, denoted by $\R_\sigma(k,w)$.
Note that many schemes \cite{Edg21,Sah21,KP24,ZKM20} make use of minimizers with a (pseudo)random order, and both cases $w\le k$ and $w>k$ are important for applications \cite{li2018minimap2,wood2014kraken,deorowicz2015kmc}.
Typically, a scheme chooses $\rho$ to be the $\le$ order on hash values of $k$-mers for some hash function.
As the density of the chosen order impacts the size of the obtained sample, computing $\R_\sigma(k,w)$ has both theoretical and practical interest.
Schleimer et al. \cite{SWA03} showed that $\R_\sigma(k,w)\approx \frac{2}{w+1}$ if ``most'' windows have no repeated $k$-mers.
Zheng et al. \cite{ZKM20} were more precise: $\R_\sigma(k,w)= \frac{2}{w+1}+o(\frac{1}{w})$ whenever $w<\sigma^{k/(3+\varepsilon)}$, while orders of density $O(\frac{1}{w})$ exist if and only if $w=O(\sigma^k)$.
Still, many natural questions about $\R_\sigma(k,w)$ are open, and we answer some of them in this paper.

We treat $\sigma$ as a constant and focus on the dependence of density on $k$ and $w$.
In Sect.~\ref{s:general}, we describe an algorithm computing $\R_\sigma(k,w)$ in $O(k\sigma^{k+w})$ time.
Then in Sect.~\ref{s:smallw} we analyse the case $w\le k$, presenting our main results.
We prove a \emph{formula} for $\R_\sigma(k,w)$, which can be computed in just $O(w\log w)$ time independent of $k$.
Studying this formula, we describe the connection between $\R_\sigma(k,w)$, $\R_\sigma(k+1,w)$, and $\R_\sigma(k,w+1)$; in particular, we show that $\R_\sigma(k,w)<\frac{2}{w+1}$ for almost all pairs $w,k$ with $w\le k$.
The notion of \emph{major run}, playing the key role in obtaining these results, can be of independent interest.
In Sect.~\ref{s:bigw} we briefly consider the case $w\gg k$, where $\R_\sigma(k,w)$ approaches its infimum $\sigma^{-k}$.
The paper ends with a discussion in Sect.~\ref{s:disc}.

\paragraph{Notation and definitions.}
In what follows, $\Sigma, \sigma,k$, and $w$ denote, respectively, the alphabet $\{0,\ldots,\sigma{-}1\}$, its size, the length of the marker substrings ($k$-mers) and the number of $k$-mers in a window.
A string $s$ over $\Sigma$ is a sequence of characters $s=s[1]s[2]\cdots s[|s|]$, where $|s|$ denotes the length of $s$.
The \emph{reversal} of $s$ is $\lvec s= s[|s|]\cdots s[2]s[1]$.
If $\le i\le j \le |s|$, we call $s[i..j]=s[i]s[i+1]\ldots s[j]$ a \emph{substring} of $s$. 
It is a \emph{prefix} if $i=1$  and a \emph{suffix} if $j=|s|$.
A \emph{repeat} is a pair of equal substrings in a string.
A \emph{$k$-string} is a string of length $k$; we similarly write \emph{$k$-prefix}, \emph{$k$-suffix}, \emph{$k$-substring}, and \emph{$k$-repeat}.

An integer $p$ is a \emph{period} of $s$ if $s[1..|s|-p]=s[p{+}1..|s|]$.
If moreover $p$ is the minimal period of $s$ and $p\le |s|/2$, then $s$ is called \emph{$p$-periodic}.
A string $s$ is \emph{primitive} if it is \emph{not} $p$-periodic for every divisor $p$ of $|s|$.
A $p$-periodic substring of $s$ is a \emph{run} (in $s$), if it is not a part of a longer $p$-periodic substring of $s$.

For a set $F$ of strings, its \emph{prefix tree} contains all prefixes of strings from $F$ as vertices and all pairs $(u,ua)$, where $a\in\Sigma$, as directed edges.
For a string $s$, its \emph{suffix tree} is the prefix tree of the set of all suffixes of $s$, additionally compressed by replacing each maximal non-branching path by a single edge.

By $\perm(\sigma,k)$ we denote the set of all permutations of all $\sigma$-ary $k$-strings.

\paragraph{Density of minimizers.} 
Let $\S_\rho$ be a minimizer for certain $\sigma,k,w$, and consider computing the markup of a string $S$.
Suppose the scheme processes windows left to right, and in the current window containing a substring $ub$, where $b\in\Sigma$, a new position is marked.
This happens either if the $k$-suffix of $ub$ is its unique minimal $k$-mer (recall that the ties are broken to the left) or if the substring $au$ from the previous window has its $k$-prefix as the minimal $k$-mer.
One can restate this condition as ``the minimal $k$-mer of the substring $aub$ is either its prefix or its unique suffix''.
A $(k+w)$-substring with this property is called a \emph{gamechanger} (also known as \emph{charged context}~\cite{ZKM20}).
One can see that the density $D_\rho$ of $\S_\rho$, defined in the introduction, equals the fraction of gamechangers among all $\sigma$-ary $(k{+}w)$-strings. 
This gives us finite and computationally efficient definition of $D_\rho$.

The \emph{density factor} is the density multiplied by the factor of $(w+1)$.
Such a normalization allows one to compare the density over a range of window sizes.
We write $\F_\sigma(k,w)= (w+1)\R_\sigma(k,w)$.

\section{Computing Expected Density in The General Case}
\label{s:general}

In order to compute the value $\R_\sigma(k,w)$ efficiently, we need a more efficient representation for it.
Let $P_{\sigma,k}(v)$ be the probability that a $(w+k)$-string $v$ is a gamechanger for a randomly chosen order $\rho\in\perm(\sigma,k)$.
We need two lemmas.
\begin{lemma} \label{l:probgc}
    If $v$ contains $t$ distinct $k$-mers, then $P_{\sigma,k}(v)=\frac{2}{t}$ if the $k$-suffix of $v$ has no other occurrences in $v$ and $P_{\sigma,k}(v)=\frac{1}{t}$ otherwise.
\end{lemma}
\begin{proof}
    As $v$ has $t$ distinct $k$-mers, the probability that the $k$-prefix of $v$ is minimal among these $k$-mers is $\frac{1}{t}$, and the same applies to the $k$-suffix.
    Now the claim follows from the definition of gamechanger. \qed
\end{proof}
\begin{lemma} \label{l:randens}
    One has $\R_\sigma(k,w) = \frac{1}{\sigma^{w+k}}\sum_{v\in\Sigma^{w+k}}P_{\sigma,k}(v)$.
\end{lemma}
\begin{proof}
    Given $v\in\Sigma^{w+k}$, $\rho\in\perm(\sigma,k)$, let $I(\rho,v)$ be $1$ if $v$ is a gamechanger according to $\rho$ and $0$ otherwise.
    From definitions, $D_\rho=\frac{1}{\sigma^{w+k}}\sum_{v\in\Sigma^{w+k}}I(\rho,v)$ and $P_{\sigma,k}(v)=\frac{1}{|\perm(\sigma,k)|}\sum_{\rho\in\perm(\sigma,k)} I(\rho,v)$.
    Then    $\R_\sigma(k,w) = 
    \frac{\sum_{\rho\in\perm(\sigma,k)} D_\rho}{|\perm(\sigma,k)|} =     \frac{\sum_{\rho\in\perm(\sigma,k)}\sum_{v\in\Sigma^{w+k}}I(\rho,v)}{|\perm(\sigma,k)|\cdot\sigma^{w+k}} =
    \frac{1}{\sigma^{w+k}}\sum_{v\in\Sigma^{w+k}}P_{\sigma,k}(v)$.\qed
\end{proof}
Lemma~\ref{l:randens} allows one to compute $\R_\sigma(k,w)$ without iterating over a huge set $\perm(\sigma,k)$: indeed, it suffices to compute the probability $P_{\sigma,k}(v)$ for every $(w+k)$-string $v$. 
By Lemma~\ref{l:probgc}, for the string $v$ we need just to count distinct $k$-mers in it.
This can be done in $O(wk)$ time naively or in $O(w+k)$ time using, e.g., the suffix tree of $v$.
The following theorem shows that we can do better, spending just $O(k)$ amortized time per string.

\begin{theorem} \label{t:compute_rand}
    The density $\R_\sigma(k,w)$ can be computed in $O(k\sigma^{w+k})$ time and $O(w+k)$ space.
\end{theorem}
\begin{proof}
    Let $\kmers(v)$ be the number of distinct $k$-mers in the string $v$.
    The idea of fast computation of $\kmers(v)$ for all $(k+w)$-strings is as follows.

    Let $\T$ be the prefix tree of the set $\Sigma^{w+k}$.
    We perform the depth-first traversal of $\T$, maintaining the suffix tree of the current node.
    Descending from a node $u$ to its child $ua$, we update the suffix tree $\ST(u)$ to $\ST(ua)$ and in the process see whether the $k$-suffix of $ua$ is a substring of $u$.
    Thus we obtain the number $\kmers(ua)$ from $\kmers(u)$.
    When ascending back from $ua$ to $u$, we revert the changes, restoring $\ST(u)$.
    Traversing the whole tree $\T$, we get $\kmers(v)$ for each leaf $v$, which is exactly what we need.
    If restoring $\ST(u)$ from $\ST(ua)$ is not slower than updating $\ST(u)$ to $\ST(ua)$, then the time complexity of our scheme is $O(t\sigma^{w+k})$, where $t$ is the worst-case time for one iteration of the suffix tree algorithm.
    We achieve $t=k$ with a version of  Weiner's algorithm \cite{Weiner73}. 
    The details are as follows.

    Building $\ST(u)$, Weiner's algorithm processes suffixes of $u$ in the order of increasing length, thus effectively reading $u$ right to left.    
    As the depth-first traversal of $\T$ corresponds to reading strings left to right, we build suffix trees of reversed strings (this suits our purposes as $\kmers(u)=\kmers(\lvec u)$).
    That is, we maintain the tree $\ST(u)$ when the current node is $\lvec u$  and update it to $\ST(au)$ on descent from $\lvec u$ to $\lvec ua$.

    In order to build $\ST(au)$ from $\ST(u)$, Weiner's algorithm finds the longest prefix $u'$ of $u$ such that $au'$ is a substring of $u$ and then adds the new leaf $au$ as a child of $au'$ (creating a node for $au'$ if it is located in the middle of an edge of $\ST(u)$).
    Note that $\kmers(au)=\kmers(u)$ if $|au'|\ge k$ and $\kmers(au)=\kmers(u)+1$ otherwise.
    After adding the new leaf, a constant number of links is added/updated (assuming a constant-size alphabet, as in our case).
    It suffices to store the links to the nodes $au'$ and $u'$ to revert the changes in constant time.
    Therefore, all information we need to store during the traversal of $\T$ is the following.
    If $\lvec u$ is the current node in $\T$, then for each prefix $x$ of $\lvec u$, including $\lvec u$ itself, we store $\kmers(x)$ and the links to two nodes in $\ST(\lvec x)$.
    Hence the overall memory usage is $O(w+k)$, as required.
    Now consider the running time.
    
    Over a constant-size alphabet, Weiner's algorithm has worst-case running time $t$ for one iteration, where $t$ is the depth of the current tree (and the total running time is linear).
    As we need $t\le k$ to prove the time bound stated in the theorem, we implement one last trick: instead of building the exact suffix tree, we construct its ``truncated'' version, cutting off all nodes $u$ with $|u|>k$.
    This can be done by a simple modification of Weiner's algorithm: for each suffix of length $>k$ we add to the tree not the suffix itself, but its $k$-prefix. 
    Such an iteration ends either in a newly created leaf (and then we have a new $k$-mer) or in already existing leaf (no new $k$-mer).
    The procedure during an iteration is the same as in Weiner's algorithm, with the last visited leaf being the starting point of the search.
    With this modification, we get the tree of depth bounded by $k$ and finally claim the time bound of $O(k\sigma^{w+k})$ for the presented algorithm. \qed
\end{proof}

\begin{remark}   
    The formula from Lemma~\ref{l:randens} can be used to compute the expected density only for small values of $k$ and $w$.
    For such values, $(k+w)$-strings can be viewed as $\sigma$-ary numbers of $(k+w)$ digits.
    Then it is faster to stick to computation over integers and avoid string operations and data structures.
    Instead of the suffix tree, a dictionary of $k$-mers of the current node can be maintained.
    Then the traversal of the prefix tree $\T$ can be organized in $O(1)$ amortized time per node (counting dictionary operations as $O(1)$).
\end{remark}

\section{The Case $w\le k$}
\label{s:smallw}

The expected density of a random minimizer depends on windows with $k$-repeats, as every window $v$, in which all $(w+1)$ $k$-mers are distinct, satisfies $P_{\sigma,k}(v)=\frac{2}{w+1}$.
If $w\le k$, then equal $k$-mers in a window necessarily overlap or touch, and thus create a periodic substring.
This simple observation has very strong implications; in particular, it leads to a formula for $\R_\sigma(k,w)$ and to a deep understanding of the behaviour of this function.

We start with combinatorial lemmas describing the mutual location of all repeated $k$-mers in a window.
Let $x$ be a $p$-periodic run in a string $v$.
We call $x$ a \emph{major run}, if $|x|\ge |v|/2 + p$.
\begin{lemma} \label{l:major}
    A string contains at most one major run.
\end{lemma}
\begin{proof}
    Aiming at a contradiction, assume that $v$ has $\rho_1$-periodic major run $x_1$ and $\rho_2$-periodic major run $x_2$.
    As their total length is at least $|v|+p_1+p_2$, they overlap in $v$ by a substring $y$ of length at least $p_1+p_2$.
    If $p_1=p_2$, then $x_1$ and $x_2$ form a $p_1$-periodic substring that contains both of them; this contradicts the definition of run.
    If $p_1\ne p_2$, then $y$ has periods $p_1$ and $p_2$ and then has the period $p=\gcd(p_1,p_2)$ by the Fine--Wilf periodicity lemma \cite{FW65}.
    Since $p$ divides $p_1$, the run $x_1$ has a $p_1$-substring that is an integer power of a $p$-substring.
    Hence $x_1$ is $p$-periodic.
    By the same argument, $x_2$ is also $p$-periodic.
    As $p<\max\{p_1,p_2\}$, we get a contradiction with our assumption on the periods of $x_1$ and $x_2$. \qed
\end{proof}
\begin{lemma} \label{l:repeated_kmers}
    Let $k\ge w$ and let a $(w+k)$-string $v$ have a repeated $k$-mer.
    Then
    \begin{enumerate}[(i)]
        \item $v$ has a major run of length at least $p+k$, where $p$ is its period; \label{stmt1}
        \item all occurrences of repeated $k$-mers in $v$ are inside the major run; \label{stmt2}
        \item $v$ has $w-i$ distinct $k$-mers, where the major run is $p$-periodic and has length $p+k+i$. \label{stmt3}
    \end{enumerate}
\end{lemma}
\begin{proof}
    \eqref{stmt1} Suppose that a $k$-mer $u$ occurs in $v$ at positions $i$ and $j>i$.
    Since $k\ge |v|/2$, these two occurrences either overlap or touch, and thus $v[i..j+k-1]$ is a $p$-periodic substring of length $j-i+k$, where $p\le j-i$ (the case $p<j-i$ takes place if there is a third occurrence of $u$ between the two considered). 
    This substring can be extended to a $p$-periodic run in $v$; the length of this run is at least $p+k$, so it is major by definition.

    \eqref{stmt2} The above procedure (start with any two occurrences of one $k$-mer and extend the obtained periodic substring to a run) results in a major run, and this run is unique by Lemma~\ref{l:major}.
    
    \eqref{stmt3} The word $v$ has $w+1$ $k$-mers in total.
    By \eqref{stmt2}, only the $k$-mers inside its major run $x$ can repeat.
    By definition of $p$-periodic, $p$ is the minimum period of $x$, so $x$ contains exactly $p$ distinct $k$-mers among its total of $p+i+1$ $k$-mers.
    The statement now follows. \qed
\end{proof}

\subsection{The Formula for $\R_\sigma(k,w)$}

Let $\rep_{\sigma,k,w}$ be the set of all $\sigma$-ary $(k+w)$-strings with $k$-repeats.
By Lemma~\ref{l:probgc}, the probability $P_{\sigma,k}(v)$ equals $\frac{2}{w+1}$ for every $v\in\Sigma^{w+k}\setminus \rep_{\sigma,k,w}$.
Then by Lemma~\ref{l:randens} we have
\begin{equation} \label{e:R}
    \R_\sigma(k,w) = \frac{2}{w+1} + \frac{1}{\sigma^{w+k}}\underbrace{\Bigg(\sum\nolimits_{v\in\rep_{\sigma,k,w}}P_{\sigma,k}(v) - \frac{2}{w+1}\big|\rep_{\sigma,k,w}\big|\Bigg)}_{\dev_\sigma(k,w)}
\end{equation}
The expression in parentheses in \eqref{e:R} shows how far is the density from the value $\frac{2}{w+1}$.
We refer to this expression as \emph{deviation} (of the random order) and denote it by $\dev_\sigma(k,w)$.
Our aim is to design, in the case $w\le k$, a formula for $\dev_\sigma(k,w)$ and thus for $\R_\sigma(k,w)$.
Let $\prim_\sigma(n)$ denote the number of $\sigma$-ary primitive words of length $n$.

\begin{lemma} \label{l:rep}
    If $w\le k$, then
\begin{equation} \label{e:Rep}
    \big|\rep_{\sigma,k,w}\big| = \sum\nolimits_{p=1}^w \prim_\sigma(p)\sigma^{w-p}\big(w-p+1-\tfrac{w-p}{\sigma}\big)
\end{equation}
\end{lemma}
\begin{proof}
    By Lemma~\ref{l:repeated_kmers}\eqref{stmt1}, every string $v\in\rep_{\sigma,k,w}$ contains a major run, say $v'$ of period $q$, and $|v'|\ge q+k$. 
    On the other hand, if a $(k+w)$-string $v$ contains a $p$-periodic substring of length $p+k$, then the $k$-prefix and $k$-suffix of this substring are equal, and hence $v\in\rep_{\sigma,k,w}$.
    Therefore, to prove the lemma we need to count, for each period $p=1,\ldots,w$, the number of $(k+w)$-strings containing a $p$-periodic run of length $\ge p+k$.
    Each such string $v$ can be uniquely represented as $v = v_1v_2v_3v_4$, where $|v_2|=k$, $|v_3|=p$, and $v_2v_3$ is a suffix of the major run.
    Note that $v_3$ is primitive (otherwise, the run would have a smaller period) and $v_2$ is uniquely determined by $v_3$ due to periodicity.
    Further, there is no restrictions on $v_1$; if $|v_4|=0$, we have
    $\sigma^{w-p}$ options for $v_1$, to the total of 
    $\prim_\sigma(p)\cdot\sigma^{w-p}$ options for $v$.
    Now let $|v_4|>0$.
    As the major run ends with $v_3$, one has $v_3[1]\ne v_4[1]$; there are no other restrictions for $v_4$.
    Since the number of options for a non-zero length of $v_4$ is $w-p$, we have $\prim_\sigma(p)\cdot(\sigma-1)\sigma^{w-p-1}(w-p)$ options for $v$ with nonempty $v_4$.
    Adding the numbers obtained for empty and nonempty $v_4$, we obtain exactly the term for $p$ in \eqref{e:Rep}. \qed
\end{proof}

\begin{lemma} \label{l:dev}
    If $w\le k$, then
\begin{multline} \label{e:sumprob}
    \sum_{v\in\rep_{\sigma,k,w}}\!\!\!P_{\sigma,k}(v) =  \\
    \sum_{t=1}^w \tfrac{1}{t}{\cdot}\Big(\prim_\sigma(t)+\sum_{p=1}^{t-1}\prim_\sigma(p)\sigma^{t-p}{\cdot}\Big(2t-2p+1-\tfrac{4t-4p-1}{\sigma}+ \tfrac{2t-2p-2}{\sigma^2}\Big)\Big)
\end{multline}
\end{lemma}
\begin{proof}
    As in Lemma~\ref{l:rep}, we view elements of $\rep_{\sigma,k,w}$ as strings containing a $p$-periodic run of length $\ge p+k$, for some $p=1,\ldots,w$.
    But unlike Lemma~\ref{l:rep}, here we need to count each string $v$ with the weight $P_{\sigma,k}(v)$ computed by Lemma~\ref{l:probgc}.
    This weight depends on the length of the run (Lemma~\ref{l:repeated_kmers}) and its location (whether it is a suffix of $v$ or not).
    Let us count $(w+k)$-strings with $t$ distinct $k$-mers.
    Such a string can be decomposed as $v=v_1v_2v_3v_4$, where
    $v_2v_3$ is the major run and $|v_3|=p$.
    Then $|v_2|=w+k-t$ by Lemma~\ref{l:repeated_kmers}\eqref{stmt3}.
    This implies $|v_1|+|v_4|=t-p$.

    First consider the case where the major run is a suffix of $v$ and hence $P_{\sigma,k}(v)=\frac{1}{t}$ by Lemma~\ref{l:probgc}.
    There are $\prim_\sigma(p)$ options for the run (for every $v_3$, $v_2$ is unique due to periodicity).
    As $|v_4|=0$, it remains to consider $v_1$.
    If $p=t$, then $v_1$ is empty, and if $p<t$, then there are $\sigma^{t-p-1}(\sigma-1)$ options for $v_1$, as the last letter of $v_1$ breaks the period of the run.
    In total, we have $\prim_\sigma(t)+ \sum_{p=1}^{t-1}\prim_\sigma(p)\sigma^{t-p-1}(\sigma-1)$ strings of weight $\frac{1}{t}$.

    Now let $|v_4|>0$ and thus $p<t$ and also $P_{\sigma,k}(v)=\frac{2}{t}$ by Lemma~\ref{l:probgc}.
    If $|v_1|=0$, we get, symmetric to the above, $\sum_{p=1}^{t-1}\prim_\sigma(p)\sigma^{t-p-1}(\sigma-1)$ strings of weight $\frac{2}{t}$.
    Finally, if $|v_1|>0$, then both the last letter of $v_1$ and the first letter of $v_4$ break the period of the run.
    Then for fixed lengths of $v_1$ and $v_4$ we get, similar to the above, $\sum_{p=1}^{t-2}\prim_\sigma(p)\sigma^{t-p-2}(\sigma-1)^2$ strings of weight $\frac{2}{t}$.
    This amount should be multiplied by $(t-p-1)$ possible choices of length for $v_1$ and $v_4$.
    Adding up the numbers obtained in all three cases, we get the term for $t$ in \eqref{e:sumprob}. \qed 
\end{proof}

The definition of $\dev_\sigma(k,w)$ and Lemmas~\ref{l:rep},~\ref{l:dev} immediately imply
\begin{proposition} \label{p:dev}
    If $w\le k$, then the deviation of the random order is
\begin{multline} \label{e:dev}
    \dev_\sigma(k,w) = 
    \sum_{t=1}^w \tfrac{1}{t}{\cdot}\Big(\prim_\sigma(t)+\sum_{p=1}^{t-1}\prim_\sigma(p)\sigma^{t-p}{\cdot}\Big(2t-2p+1-\tfrac{4t-4p-1}{\sigma}+ \tfrac{2t-2p-2}{\sigma^2}\Big)\Big) \\
    - \tfrac{2}{w+1}\cdot\sum_{p=1}^w \prim_\sigma(p)\sigma^{w-p}\big(w-p+1-\tfrac{w-p}{\sigma}\big)
\end{multline}
In particular, $\dev_\sigma(k,w)$ is independent of $k$.
\end{proposition}

Substituting \eqref{e:dev} into \eqref{e:R}, we obtain the main result of this section.
\begin{theorem} \label{t:Rformula}
    If $w\le k$, the random order has the expected density $\R_\sigma(k,w) = \frac{2}{w+1} + \frac{\dev_\sigma(k,w)}{\sigma^{w+k}}$, where $ \dev_\sigma(k,w)$ is given by the formula \eqref{e:dev}.
    In particular, $\R_\sigma(k,w)$ can be computed in $O(w^2)$ time independently of $\sigma$ and $k$.
\end{theorem}
\begin{proof}
    The formula is already proved, so it remains to show the time complexity.
    Note that if a string is not primitive, then it has the form $u^m$, where $u$ is primitive and $m$ is an integer greater than 1.
    Therefore, $\prim_\sigma(p)= \sigma^p-\sum_{d\mid p}\prim_\sigma(d)$.
    To compute $\prim_\sigma(1),\ldots,\prim_\sigma(w)$ by this formula, we initialize an array of length $w$ with the values $\sigma,\sigma^2,\ldots,\sigma^w$ and process it left to right; reaching the cell $i$, we subtract its value from the values in cells $2i,3i,\ldots$. 
    The processing time is then $w+\frac{w}{2}+\frac{w}{3}+\cdots = O(w\log w)$.
    Then we compute the deviation according to \eqref{e:dev} in just $O(w)$ time by memorizing the values of the internal sum for the smaller values of $t$.
     \qed
\end{proof}

Given Theorem~\ref{t:Rformula}, we can study details of the behaviour of the function $\R_\sigma(k,w)$ in the half-quadrant defined by the inequality $w\le k$.

\subsection{From $\R_\sigma(k,w)$ to $\R_\sigma(k+1,w)$}

\begin{proposition} \label{p:vert}
    If $w\le k$, then $\R_\sigma(k+1,w)=\frac{2}{w+1}+\frac{1}{\sigma}\cdot(\R_\sigma(k,w)- \frac{2}{w+1})$.
\end{proposition}
\begin{proof}
    As we know from Proposition~\ref{p:dev}, $\dev_\sigma(k,w)=\dev_\sigma(k+1,w)$.
    Then the result is immediate from Theorem~\ref{t:Rformula}. \qed
\end{proof}
Then we immediately have
\begin{corollary}
    If $w$ is fixed and $k$ tends to infinity, the density $\R_\sigma(k,w)$ approaches $\frac{2}{w+1}$ (equivalently, the density factor $\F_\sigma(k,w)$ approaches $2$) at \emph{exact} exponential rate $\sigma$. 
\end{corollary}

The crucial fact that $\dev_\sigma(k,w)$ does not depend on $k$ looks unexpected and calls for a better explanation of its nature.
Below we explain it establishing a natural bijection between the sets $\rep_{\sigma,k,w}$ and $\rep_{\sigma,k+1,w}$.

Consider the following function $\phi$ defined on $\rep_{\sigma,k,w}$ $(w\le k)$.
Given a string $v\in \rep_{\sigma,k,w}$, let $x$ be its major run (Lemma~\ref{l:repeated_kmers}\eqref{stmt1}), let $p$ be the period of $x$, and let $a_1a_2\cdots a_p$ be the $p$-suffix of $x$.
We write $v=\ell xr$.
Then $\phi(v)$ is a $(k+w+1)$-string of the form $\ell x'r'$ such that $x'=xa_1$ and $r'$ is defined as follows.
If either $r$ is empty, or $p=1$, or $r[1]\ne a_2$, then $r'=r$; otherwise (i.e., if $r[1]=a_2$), $r'$ is obtained from $r$ by replacing $r[1]$ with $a_1$.

\begin{theorem} \label{t:bijection}
    Let $w\le k$.
    Then $\phi$ is a bijection of $\rep_{\sigma,k,w}$ onto $\rep_{\sigma,k+1,w}$ and $P_{\sigma,k}(v)=P_{\sigma,k+1}(\phi(v))$ for every $v\in \rep_{\sigma,k,w}$.
\end{theorem}
\begin{proof}
    We first prove the following claim.
    \begin{claim}
        If a string $v\in \rep_{\sigma,k,w}$ has $p$-periodic major run $x$ and $a_1a_2\cdots a_p$ is the $p$-suffix of $x$, then $xa_1$ is the major run in $\phi(v)$.
    \end{claim}
    Let $v=\ell xr$.
    Since $x$ is $p$-periodic and ends with $a_1a_2\cdots a_p$, $x'=xa_1$ is also $p$-periodic.
    Since $x'$ is preceded in $\phi(v)=\ell x'r'$ by the same string $\ell$ as $x$ in $v$, it cannot be extended to the left; it also cannot be extended to the right if $r'$ is empty.
    For nonempty $r'$ we consider two cases.
    If $r'=r$, then by the definition of $\phi$ either $p=1$ or $r[1]\ne a_2$.
    In both cases, $r'[1]=r[1]$ breaks the period $p$ in $\phi(v)$.
    If $r'\ne r$, then $p>1$ and $r'[1]=a_1$ while $r[1]=a_2$.
    We know that $r[1]\ne a_1$, because $r[1]$ breaks the period $p$ in $v$. 
    Hence $a_1\ne a_2$, and once again $r'[1]$ breaks the period $p$. 
    Therefore the $p$-periodic string $x'$ can be extended neither to the left or to the right, and thus is a run.
    Clearly, $x'$ satisfies the length condition in the definition of major run.
    Hence, the claim is proved.

\smallskip
    Now we proceed with the proof of the theorem.
    Since $|x|\ge k+p$ by the definition of major run, the $k$-prefix of $x$ repeats in $x$.
    Then the $(k+1)$-prefix of $x'$ repeats in $x'$.
    Therefore, $\phi(v)\in\rep_{\sigma,k+1,w}$.
    To prove that $\phi$ is bijective it suffices to consider an arbitrary string $v'\in\rep_{\sigma,k+1,w}$ and show that it has exactly one preimage by $\phi$.
    By Lemma~\ref{l:repeated_kmers}\eqref{stmt1}, $v'$ has a major run.

    Let $v'=\ell x'r'$, where $x'$ is the major run, and let $x'=\hat xa_1a_2\cdots a_pa_1$, where $p$ is the period of $x'$.
    We denote $x=\hat xa_1a_2\cdots a_p$.
    By the Claim and Lemma~\ref{l:major}, every preimage of $v'$ has $x$ as the major run.
    Therefore, by the definition of $\phi$, we can consider as candidate preimages of $v'$ only strings of the form $\ell xr$ where $r$ either equals $r'$ or differs from $r'$ in the first letter only. 
    In particular, if $r'$ is empty, then $\ell x$ is the only candidate preimage and clearly $\phi(\ell x)=\ell x'=v'$.
    Now let $r'=b'\hat r$, $r=b\hat r$.
    The definition of $\phi$ tells us that either $b=b'$ or $b=a_2$ and $b'=a_1$.
    Then in the case $b'\ne a_1$ the only candidate for $b$ is $b'$, and one can check that $\phi(\ell xb'\hat r)=v'$ by definition.
    For the case $b'=a_1$ we observe that $b\ne a_1$, because $x$ is a run in $\ell xr$.
    Then $b=a_2$ is the only candidate, and again, $\phi(\ell xa_2\hat r)=v'$ by definition.
    Thus we proved that $v'$ has exactly one preimage by $\phi$, and therefore $\phi$ is a bijection.
    
    It remains to prove that $\phi$ preserves probabilities to be a gamechanger. 
    Since the major runs $x$ (of $v$) and $x'$ (of $\phi(v)$) have the same period, Lemma~\ref{l:repeated_kmers}\eqref{stmt3} implies that the number of distinct $k$-mers in $v$ equals the number of distinct $(k+1)$-mers in $\phi(v)$.
    Next, note that the $k$-suffix of $x$ has at least two occurrences in $x$.
    Hence, by Lemma~\ref{l:repeated_kmers}\eqref{stmt2}, the $k$-suffix of $v$ has another occurrence in $v$ if and only if $x$ is a suffix of $v$.
    By a similar argument, the $(k+1)$-suffix of $\phi(v)$ has another occurrence in $\phi(v)$ if and only if $x'$ is a suffix of $\phi(v)$.
    By the definition of $\phi$, $x$ is a suffix of $v$ if and only if $x'$ is a suffix of $\phi(v)$.
    Therefore, $P_{\sigma,k}(v)=P_{\sigma,k+1}(\phi(v))$ by Lemma~\ref{l:probgc}. \qed
\end{proof}

\subsection{From $\R_\sigma(k,w)$ to $\R_\sigma(k,w+1)$}

Comparing the densities $\R_\sigma(k,w)$ and $\R_\sigma(k,w+1)$ we need to compare their deviations. 
As $\dev_\sigma(k,w)$ does not depend on $k$ by Proposition~\ref{p:dev}, below we write $\dev_\sigma(w)$.
Let $\Delta_\sigma(w)=\dev_\sigma(w+1)-\dev_\sigma(w)$.
We prove the following.
\begin{lemma} \label{l:delta_dev}
    For every $\sigma\ge 2$ and $w\le k-1$, one has $\Delta_\sigma(w)=\frac{S_1+S_2}{(w+1)(w+2)}$, where
\begin{gather}
\begin{split} \label{e:s1}
S_1 & = (0-w)\sigma^0\prim_\sigma(w{+}1)+(2-w)\sigma\prim_\sigma(w)+(4-w)\sigma^2\prim_\sigma(w{-}1)+\cdots\\
& \cdots +(w-2)\sigma^{w-1}\prim_\sigma(2)+w\sigma^w\prim_\sigma(1)
\end{split}\\
\begin{split} \label{e:s2}
S_2 & = (w-0)\sigma^0\prim_\sigma(w)+(w-2)\sigma\prim_\sigma(w{-}1)+ (w-4)\sigma^2\prim_\sigma(w{-}2)+\cdots\\
& \cdots +(4-w)\sigma^{w-2}\prim_\sigma(2)+
(2-w)\sigma^{w-1}\prim_\sigma(1)
\end{split}
\end{gather}
\end{lemma}

\begin{proof}
By the definition of deviation, $\Delta_\sigma(w)=\Delta_1-\Delta_2$, where
\begin{align*}
\Delta_1 &= \sum_{v\in\rep_{\sigma,k,w+1}}\hspace*{-4.5mm}P_{\sigma,k}(v) - \sum_{v\in\rep_{\sigma,k,w}}\hspace*{-3mm}P_{\sigma,k}(v),\\ \Delta_2 &= \tfrac{2}{w+2}\big|\rep_{\sigma,k,w+1}\big| - \tfrac{2}{w+1}\big|\rep_{\sigma,k,w}\big|.
\end{align*}
Note that all terms in the sum \eqref{e:sumprob} are independent of $w$. Then $\Delta_1$ equals such a term for $t=w+1$, i.e.,
$$
\Delta_1 = \tfrac{1}{w+1}\Big(\prim_\sigma(w+1)+ \sum_{p=1}^{w}\prim_\sigma(p){\cdot}\sigma^{w+1-p}\Big(2w-2p+3-\tfrac{4w-4p+3}{\sigma}+ \tfrac{2w-2p}{\sigma^2}\Big)\Big).
$$
Using \eqref{e:R} and \eqref{e:Rep}, we compute
\begin{align*}
\Delta_2 &= \tfrac{2}{w+2}\sum_{p=1}^{w+1} \prim_\sigma(p)\sigma^{w-p+1}\big(w-p+2-\tfrac{w-p+1}{\sigma}\big)\\
&\phantom{=}- \tfrac{2}{w+1}\sum_{p=1}^w \prim_\sigma(p)\sigma^{w-p}\big(w-p+1-\tfrac{w-p}{\sigma}\big) 
\\&= \tfrac{2}{w+2}\prim_\sigma(w+1) \\&\phantom{=}+ \sum_{p=1}^w \prim_\sigma(p)\sigma^{w-p}\Big( \tfrac{2(w-p+2)\sigma}{w+2}- \tfrac{2(w-p+1)}{w+2}-\tfrac{2(w-p+1)}{w+1}+\tfrac{2(w-p)}{(w+1)\sigma}\Big) 
\end{align*}
Grouping the corresponding terms in the expressions for $\Delta_1$ and $\Delta_2$, we get
\begin{multline} \label{e:s1s2_wrapped}
\Delta_\sigma(w)=\Delta_1-\Delta_2 = \Big(\tfrac{1}{w+1}- \tfrac{2}{w+2}\Big)\prim_\sigma(w+1) \\+ \sum_{p=1}^w \prim_\sigma(p)\sigma^{w-p}\Big( \tfrac{(2w-2p+3)\sigma}{w+1}- 
\tfrac{(2w-2p+4)\sigma}{w+2} + 
\tfrac{2w-2p+2}{w+2}-\tfrac{2w-2p+1}{w+1}+0\Big) \\
= \tfrac{-w\prim_\sigma(w+1)+\sum_{p=1}^w \prim_\sigma(p)\sigma^{w-p}((w-2p+2)\sigma - w +2p)}{(w+1)(w+2)}
\end{multline}
Unwrapping the sum, we get exactly $S_1+S_2$ in the numerator. \qed
\end{proof}

To estimate $\Delta_\sigma(w)$ from \eqref{e:s1} and \eqref{e:s2}, we need to evaluate $\prim_\sigma$.
We recall (see, e.g., \cite{Lothaire_1997}) that \begin{equation}\label{e:prim}
    \prim_\sigma(p) = \sum\nolimits_{d\mid p} \mu(d)\sigma^{p/d},
\end{equation}
where the \emph{M\"obius function} $\mu(n)$ is defined as follows. 
If $n$ is \emph{square-free}, i.e., a product of $t$ distinct primes for some $t$, then $\mu(n)=(-1)^t$, including $\mu(1)=(-1)^0=1$.
Otherwise, $\mu(n)=0$.
For example, $\prim_\sigma(1)=\sigma$, $\prim_\sigma(2)=\sigma^2-\sigma$, $\prim_\sigma(4)=\sigma^4-\sigma^2$, $\prim_\sigma(60)=\sigma^{60}-\sigma^{30}-\sigma^{20}-\sigma^{12}+\sigma^{10}+\sigma^6+\sigma^4-\sigma^2$.

Substituting the values of $\prim_\sigma$ into \eqref{e:s1} and \eqref{e:s2}, one can see that for any fixed $w$, $\Delta_\sigma(w)$ is a polynomial in $\sigma$; in particular,
\begin{align}\label{e:smallw}
        \Delta_\sigma(1)=\tfrac{\sigma}{3}, 
        \Delta_\sigma(2)=\tfrac{\sigma^2}{6}, \Delta_\sigma(3)=\tfrac{2\sigma^3+3\sigma^2-3\sigma}{20}, \Delta_\sigma(4)=\tfrac{2\sigma^4+2\sigma^3-6\sigma^2+4\sigma}{30}
\end{align}
The following lemma is proved by direct computation.
\begin{lemma} \label{l:delta_smallw}
    For every $\sigma\ge 2$ and every $w\le \min\{10,k-1\}$, $\Delta_\sigma(w)>0$.
\end{lemma}
\begin{proof}[of Lemma~\ref{l:delta_smallw}]
Similar to \eqref{e:smallw}, we can unwrap \eqref{e:s1} and \eqref{e:s2} to get 
\begin{align*}
    \Delta_\sigma(5)&=\tfrac{1}{42}(2\sigma^5+\sigma^4+\sigma^3+8\sigma^2-10\sigma) \\
    \Delta_\sigma(6)&=\tfrac{1}{56}(2\sigma^6+2\sigma^4-14\sigma^2+12\sigma) \\
    \Delta_\sigma(7)&=\tfrac{1}{72}(2\sigma^7-\sigma^6+3\sigma^5+6\sigma^4-11\sigma^3+10\sigma^2-5\sigma) \\
    \Delta_\sigma(8)&=\tfrac{1}{90}(2\sigma^8-2\sigma^7+4\sigma^6+4\sigma^5-16\sigma^4+16\sigma^3-6\sigma^2) \\
    \Delta_\sigma(9)&=\tfrac{1}{110}(2\sigma^9-3\sigma^8+5\sigma^7+2\sigma^6-3\sigma^5+13\sigma^4-14\sigma^3+9\sigma^2-9\sigma) \\
    \Delta_\sigma(10)&=\tfrac{1}{132}(2\sigma^{10}-4\sigma^9+6\sigma^8-12\sigma^4+8\sigma^3-18\sigma^2+20\sigma) 
\end{align*}

With the use of basic calculus, each of the polynomials $\Delta_\sigma(1),\ldots,\Delta_\sigma(10)$ can be proved positive for all $\sigma\ge 2$. \qed
\end{proof}

To the contrast, as $w$ grows, $\Delta_\sigma(w)$ becomes negative and approaches $-\infty$.
\begin{lemma} \label{l:delta_bigw}
    Let $\sigma\ge 2$ be fixed.
    There exist constants $C_\sigma, C_\sigma'>0$ such that $\frac{2\sigma+6-C_\sigma w}{(w+1)(w+2)}\sigma^{w-1}<\Delta_\sigma(w)<\frac{2\sigma+6-C_\sigma' w}{(w+1)(w+2)}\sigma^{w-1}$ whenever $11\le w\le k-1$.
\end{lemma}

\begin{proof}
Assume $w\ge 11$ and denote $\tilde{\Delta}_\sigma(w)=(w+1)(w+2)\Delta_\sigma(w)$ for convenience.
By Lemma~\ref{l:delta_dev} and formula \eqref{e:prim}, we have $\tilde{\Delta}_\sigma(w)=\sum_{t=1}^{w+1} c_t(w)\sigma^t$, where each coefficient $c_t(w)$ is a sum of $O(w)$ numbers, each of absolute value at most $w$.
We first compute the three leading coefficients of $\tilde{\Delta}_\sigma(w)$.

The exponent $\sigma^{w+1}$ does not appear in $S_2$ \eqref{e:s2}, while in $S_1$ \eqref{e:s1} it appears with the coefficient $c_{w+1}(w)= -w+(2-w)+\cdots+(w-2)+w=0$.
Next, $\sigma^{w}$ appears in $S_2$ with the coefficient $w+(w-2)+(w-4)+\cdots+(4-w)+(2-w)=w$; in $S_1$ it appears only in the term containing $\prim_\sigma(2)=\sigma^2-\sigma$, with the coefficient $2-w$.
Hence $c_w(w)=2$.
Finally, the exponent $\sigma^{w-1}$ appears once in $S_2$ (in the term of $\prim_\sigma(2)$ with the coefficient $w-4$) and twice in $S_1$ (in the terms of $\prim_\sigma(4)$ and $\prim_\sigma(3)$, with the coefficients $6-w$ and $4-w$ respectively).
Hence $c_{w-1}(w)=6-w$.
Therefore, the three leading terms in $\tilde{\Delta}_\sigma(w)$ sum up to $X=(2\sigma+6-w)\sigma^{w-1}$; below we define $Y,Z$ so that $\tilde{\Delta}_\sigma(w)=X+Y+Z$.

All lower terms in $\tilde{\Delta}_\sigma(w)$ are due to the monomials $\pm \sigma^{p-m}$ appearing in the expansion \eqref{e:prim} of $\prim_\sigma(p)$ for certain $p>m\ge 2$.
Every such monomial contributes to the coefficients at $\sigma^{w+1-m}$ and $\sigma^{w-m}$ (see \eqref{e:s1s2_wrapped}).
If $\prim_\sigma(p)$ contains $\pm \sigma^{p-m}$, then $p=(p-m)d$ for a square-free $d$.
Hence $p=\frac{md}{d-1}$, implying that $d-1$ divides $m$ and the maximum value of $p$ is $2m$.
In particular, $p=3,4$ for $m=2$; $p=6$ for $m=3$; $p=5,6,8$ for $m=4$; and $p=6,10$ for $m=5$.
From \eqref{e:s1}, \eqref{e:s2} (or from \eqref{e:s1s2_wrapped}) we see that the monomials $\pm \sigma^{p-m}$ with $m\in\{2,3,4,5\}$ contribute $Y=(w-4)\sigma^{w-2}+(20-2w)\sigma^{w-3}+(3w-30)\sigma^{w-4}-8\sigma^{w-5}$ to $\tilde{\Delta}_\sigma(w)$ (the contribution of $-\sigma^{p-2}$ to the coefficient of $\sigma^{w-1}$ is included into $X$). 

We rewrite $Y=(y_1w-y_2)\sigma^{w-1}$, where $y_1=\frac{1}{\sigma}-\frac{2}{\sigma^2}+\frac{3}{\sigma^3}>0$ and $y_2=\frac{4}{\sigma}-\frac{20}{\sigma^2}+\frac{30}{\sigma^3}+\frac{8}{\sigma^4}>0$.
As $w\ge11$, we have $(y_1-\frac{y_2}{11})w\sigma^{w-1}\le Y<y_1w\sigma^{w-1}$.

Let $Z$ be the sum of terms in $\tilde{\Delta}_\sigma(w)$ arising from monomials $\pm \sigma^{p-m}$ in \eqref{e:prim} with $m\ge 6$.
Then $\tilde{\Delta}_\sigma(w)=X+Y+Z$ as desired.
As mentioned above, for a fixed $m$, the values of $p$ satisfy $p=\frac{md}{d-1}$, where $d$ is square-free.
So either $p=m+1$ or $d-1\le \frac{m}{2}$ and $d-1\ne 3$.
Hence the number of options for $p$ is at most $\frac{m}{2}$.
The coefficient for $\sigma^{p-m}$ for one term of \eqref{e:s1s2_wrapped} is between $-w$ and $w$.
Therefore, we bound the absolute value of $Z$ as $|Z|\le \sum_{m=6}^\infty \frac{m}{2}w\sigma^{w+1-m}$.
Factoring $\frac{w\sigma^w}{2}$ out and substituting $x=\sigma^{-1}, t=6$ into the textbook formula $\sum_{m=t}^\infty mx^{m-1}=\frac{tx^{t-1}-(t-1)x^t}{(1-x)^2}$, we obtain $|Z|\le zw\sigma^{w-1}$, where $z=\frac{6\sigma^{-2}-5\sigma^{-3}}{2(\sigma-1)^2}$.

For any $\sigma\ge 2$, one can easily check that $y_1+z<1$ and $y_1-\frac{y_2}{11}-z>-\infty$.
Then $\tilde{\Delta}_\sigma(w)=X+Y+Z$ is between $(2\sigma+6-C_\sigma w)\sigma^{w-1}$ and $(2\sigma+6-C_\sigma'w)\sigma^{w-1}$ for some constants $C_\sigma, C_\sigma'>0$.
The lemma follows from definition of $\tilde{\Delta}_\sigma(w)$. \qed
\end{proof}

Now we describe the main features of the ``horizontal'' behaviour of the density $\R_\sigma(k,w)$.
As we compare the values for different $w$, it is convenient to formulate the result in terms of density factor $\F_\sigma(k,w)=(w+1)\R_\sigma(k,w)$.

\begin{theorem} \label{t:horiz} 
    For every $\sigma\ge 2$ there exist integers $w_\sigma,w_\sigma'\ge 11$ such that
    \begin{enumerate}[(i)]
        \item $\F_\sigma(k,w)>2$ if $2\le w\le \min\{w_\sigma,k\}$; \label{it:wsmall}
        \item $\F_\sigma(k,w)<2$ if $w_\sigma'\le w\le k$; in this case $2-\F_\sigma(k,w)=\Theta(\frac{1}{w\sigma^k})$ . \label{it:wbig}
    \end{enumerate}
\end{theorem}
\begin{proof}
    Note that if $w=1$, for any order $\rho$ we have $D_\rho=1=\tfrac{2}{w+1}$ and therefore $\F_\sigma(k,1)=2$.
    Then \eqref{it:wsmall} follows from Lemma~\ref{l:delta_smallw}. 
    Further, Lemma~\ref{l:delta_bigw} proves that $\Delta_\sigma(w)$ is negative starting from some value of $w$ and approaches $-\infty$ as $w$ grows; then the deviation becomes negative starting from some $w=w_\sigma'$, implying the first statement of \eqref{it:wbig}.
    Now note that $\Delta_\sigma(w)=\Theta(\frac{\sigma^w}{w})$, and hence the same bound works for $\dev_\sigma(w)$.
    Formula \eqref{e:R} implies the second statement of \eqref{it:wbig}. 
    \qed
\end{proof}

\begin{remark} \label{r:transit}
    Combining the bounds from Theorem~\ref{t:horiz} with experiments for small alphabets, we claim a stronger result for those alphabets.
    Namely, a \emph{single} constant separate the zones where $\F_\sigma(k,w)>2$ and $\F_\sigma(k,w)<2$ (see Tables~\ref{tbl:binary},~\ref{tbl:10ary}).
    However, we have no proof of this property for arbitrary alphabets.
\end{remark}

\begin{table}[!ht] 
\centering
\renewcommand{\arraystretch}{1.2} 
\setlength{\tabcolsep}{3pt} 
\caption{$\sigma=2$, blue cells represent $\F(k,w)>2$ and green cells are $\F(k,w)<2$. 
Moreover, for any pair $w\le k$, if $w<17$ then $\F(k,w)\ge2$ and if $w\ge 17$ then $\F(k,w)<2$.
The value in every cell is  $\log_2|\F(k,w)-2|$. }
\label{tbl:binary}
\begin{adjustbox}{width=0.97\columnwidth,center}
\begin{tabular}{c|*{15}{c}}  
    \diagbox[height=2em]{$k$}{$w$} & \textbf{2} & \textbf{3} & \textbf{4} & \textbf{5} & \textbf{\ldots} & \textbf{\ldots} & \textbf{15} & \textbf{16} & \textbf{17} & \textbf{18} & \textbf{19} & \textbf{20} & \textbf{21} & \textbf{22} & \textbf{23} \\ \hline
    \textbf{2} & \cellcolor{LightSkyBlue1}-3 &  &  &  &  &  &  &  &  &  &  &  &  &  &  \\ 
    \textbf{3} & \cellcolor{LightSkyBlue1}-4 & \cellcolor{LightSkyBlue1}-3.6 &  &  &  &  &  &  &  &  &  &  &  &  &  \\ 
    \textbf{4} & \cellcolor{LightSkyBlue1}-5 & \cellcolor{LightSkyBlue1}-4.6 & \cellcolor{LightSkyBlue1}-4.4 &  &  &  &  &  &  &  &  &  &  &  &  \\ 
    \textbf{5} & \cellcolor{LightSkyBlue1}-6 & \cellcolor{LightSkyBlue1}-5.6 & \cellcolor{LightSkyBlue1}-5.4 & \cellcolor{LightSkyBlue1}-5.6 &  &  &  &  &  &  &  &  &  &  &  \\ 
    \textbf{\vdots} & \cellcolor{LightSkyBlue1}\vdots & \cellcolor{LightSkyBlue1}\vdots &\cellcolor{LightSkyBlue1} \vdots &\cellcolor{LightSkyBlue1} \vdots &\cellcolor{LightSkyBlue1} $\ddots$ & &  &  &  &  &  &  &  &  &  \\ 
    \textbf{\vdots} & \cellcolor{LightSkyBlue1}\vdots & \cellcolor{LightSkyBlue1}\vdots & \cellcolor{LightSkyBlue1}\vdots & \cellcolor{LightSkyBlue1}\vdots & \cellcolor{LightSkyBlue1}$\ddots$ &\cellcolor{LightSkyBlue1} $\ddots$ &  &  &  &  &  &  &  &  &  \\ 
    \textbf{15} & \cellcolor{LightSkyBlue1}-16 & \cellcolor{LightSkyBlue1}-15.6 & \cellcolor{LightSkyBlue1}-15.4 & \cellcolor{LightSkyBlue1}-15.6 & \cellcolor{LightSkyBlue1}\ldots &\cellcolor{LightSkyBlue1} \ldots & \cellcolor{LightSkyBlue1}-19.2 &  &  &  &  &  &  &  &  \\ 
    \textbf{16} & \cellcolor{LightSkyBlue1}-17 & \cellcolor{LightSkyBlue1}-16.6 & \cellcolor{LightSkyBlue1}-16.4 & \cellcolor{LightSkyBlue1}-16.6 & \cellcolor{LightSkyBlue1}\ldots & \cellcolor{LightSkyBlue1}\ldots & \cellcolor{LightSkyBlue1}-20.2 & \cellcolor{LightSkyBlue1}-21.3 &  &  &  &  &  &  &  \\ 
    \textbf{17} & \cellcolor{LightSkyBlue1}-18 & \cellcolor{LightSkyBlue1}-17.6 & \cellcolor{LightSkyBlue1}-17.4 & \cellcolor{LightSkyBlue1}-17.6 & \cellcolor{LightSkyBlue1}\ldots & \cellcolor{LightSkyBlue1}\ldots & \cellcolor{LightSkyBlue1}-21.2 & \cellcolor{LightSkyBlue1}-22.3 & \cellcolor{DarkSeaGreen2}-25.5 &  &  &  &  &  &  \\ 
    \textbf{18} & \cellcolor{LightSkyBlue1}-19 & \cellcolor{LightSkyBlue1}-18.6 & \cellcolor{LightSkyBlue1}-18.4 & \cellcolor{LightSkyBlue1}-18.6 & \cellcolor{LightSkyBlue1}\ldots & \cellcolor{LightSkyBlue1}\ldots & \cellcolor{LightSkyBlue1}-22.2 & \cellcolor{LightSkyBlue1}-23.3 & \cellcolor{DarkSeaGreen2}-26.5 & \cellcolor{DarkSeaGreen2}-23.3 &  &  &  &  &  \\ 
    \textbf{19} & \cellcolor{LightSkyBlue1}-20 & \cellcolor{LightSkyBlue1}-19.6 & \cellcolor{LightSkyBlue1}-19.4 & \cellcolor{LightSkyBlue1}-19.6 & \cellcolor{LightSkyBlue1}\ldots & \cellcolor{LightSkyBlue1}\ldots & \cellcolor{LightSkyBlue1}-23.2 & \cellcolor{LightSkyBlue1}-24.3 & \cellcolor{DarkSeaGreen2}-27.5 & \cellcolor{DarkSeaGreen2}-24.3 & \cellcolor{DarkSeaGreen2}-23.4 &  &  &  &  \\ 
    \textbf{20} & \cellcolor{LightSkyBlue1}-21 & \cellcolor{LightSkyBlue1}-20.6 & \cellcolor{LightSkyBlue1}-20.4 & \cellcolor{LightSkyBlue1}-20.6 & \cellcolor{LightSkyBlue1}\ldots & \cellcolor{LightSkyBlue1}\ldots & \cellcolor{LightSkyBlue1}-24.2 & \cellcolor{LightSkyBlue1}-25.3 & \cellcolor{DarkSeaGreen2}-28.5 & \cellcolor{DarkSeaGreen2}-25.3 & \cellcolor{DarkSeaGreen2}-24.4 & \cellcolor{DarkSeaGreen2}-23.9 &  &  &  \\ 
    \textbf{21} & \cellcolor{LightSkyBlue1}-22 & \cellcolor{LightSkyBlue1}-21.6 & \cellcolor{LightSkyBlue1}-21.4 & \cellcolor{LightSkyBlue1}-21.6 & \cellcolor{LightSkyBlue1}\ldots & \cellcolor{LightSkyBlue1}\ldots & \cellcolor{LightSkyBlue1}-25.2 & \cellcolor{LightSkyBlue1}-26.3 & \cellcolor{DarkSeaGreen2}-29.5 & \cellcolor{DarkSeaGreen2}-26.3 & \cellcolor{DarkSeaGreen2}-25.4 & \cellcolor{DarkSeaGreen2}-24.9 & \cellcolor{DarkSeaGreen2}-24.6 &  &  \\ 
    \textbf{22} & \cellcolor{LightSkyBlue1}-23 & \cellcolor{LightSkyBlue1}-22.6 & \cellcolor{LightSkyBlue1}-22.4 & \cellcolor{LightSkyBlue1}-22.6 & \cellcolor{LightSkyBlue1}\ldots & \cellcolor{LightSkyBlue1}\ldots & \cellcolor{LightSkyBlue1}-26.2 & \cellcolor{LightSkyBlue1}-27.3 & \cellcolor{DarkSeaGreen2}-30.5 & \cellcolor{DarkSeaGreen2}-27.3 & \cellcolor{DarkSeaGreen2}-26.4 & \cellcolor{DarkSeaGreen2}-25.9 & \cellcolor{DarkSeaGreen2}-25.6 & \cellcolor{DarkSeaGreen2}-25.4 &  \\ 
    \textbf{23} & \cellcolor{LightSkyBlue1}-24 & \cellcolor{LightSkyBlue1}-23.6 & \cellcolor{LightSkyBlue1}-23.4 & \cellcolor{LightSkyBlue1}-23.6 & \cellcolor{LightSkyBlue1}\ldots & \cellcolor{LightSkyBlue1}\ldots & \cellcolor{LightSkyBlue1}-27.2 & \cellcolor{LightSkyBlue1}-28.3 & \cellcolor{DarkSeaGreen2}-31.5 & \cellcolor{DarkSeaGreen2}-28.3 & \cellcolor{DarkSeaGreen2}-27.4 & \cellcolor{DarkSeaGreen2}-26.9 & \cellcolor{DarkSeaGreen2}-26.6 & \cellcolor{DarkSeaGreen2}-26.4 & \cellcolor{DarkSeaGreen2}-26.2 \\ 
\end{tabular}
\end{adjustbox}
\end{table}

\begin{table}[!bh]
\renewcommand{\arraystretch}{1.2} 
\caption{$\sigma=10$, blue cells represent $\F(k,w)>2$ and green cells are $\F(k,w)<2$. 
Moreover, for any pair $w\le k$, if $w<30$ then $\F(k,w)\ge2$ and if $w\ge 30$ then $\F(k,w)<2$.
The value in every cell is  $\log_{10}|\F(k,w)-2|$. }
\label{tbl:10ary}
\begin{adjustbox}{width=0.97\columnwidth,center}
\begin{tabular}{c|*{16}{c}}
    \diagbox[height=2em]{$k$}{$w$} & \textbf{2} & \textbf{3} & \textbf{4} & \textbf{5} & \textbf{\ldots} & \textbf{\ldots} & \textbf{28} & \textbf{29} & \textbf{30} & \textbf{31} & \textbf{32} & \textbf{33} & \textbf{34} & \textbf{35} & \textbf{36} & \textbf{37} \\ \hline
    \textbf{2} & \cellcolor{LightSkyBlue1}-3 &  &  &  &  &  &  &  &  &  &  &  &  &  &  &  \\ 
    \textbf{3} & \cellcolor{LightSkyBlue1}-4 & \cellcolor{LightSkyBlue1}-4.1 &  &  &  &  &  &  &  &  &  &  &  &  &  &  \\ 
    \textbf{4} & \cellcolor{LightSkyBlue1}-5 & \cellcolor{LightSkyBlue1}-5.1 & \cellcolor{LightSkyBlue1}-5.2 &  &  &  &  &  &  &  &  &  &  &  &  &  \\ 
    \textbf{5} & \cellcolor{LightSkyBlue1}-6 & \cellcolor{LightSkyBlue1}-6.1 & \cellcolor{LightSkyBlue1}-6.2 & \cellcolor{LightSkyBlue1}-6.3 &  &  &  &  &  &  &  &  &  &  &  &  \\ 
    \textbf{\vdots} & \cellcolor{LightSkyBlue1}\vdots & \cellcolor{LightSkyBlue1}\vdots & \cellcolor{LightSkyBlue1}\vdots & \cellcolor{LightSkyBlue1}\vdots & \cellcolor{LightSkyBlue1}$\ddots$ &  &  &  &  &  &  &  &  &  &  &  \\ 
    \textbf{\vdots} & \cellcolor{LightSkyBlue1}\vdots & \cellcolor{LightSkyBlue1}\vdots & \cellcolor{LightSkyBlue1}\vdots & \cellcolor{LightSkyBlue1}\vdots & \cellcolor{LightSkyBlue1}$\ddots$ &\cellcolor{LightSkyBlue1} $\ddots$ &  &  &  &  &  &  &  &  &  &  \\ 
    \textbf{28} & \cellcolor{LightSkyBlue1}-29 & \cellcolor{LightSkyBlue1}-29.1 & \cellcolor{LightSkyBlue1}-29.2 & \cellcolor{LightSkyBlue1}-29.3 & \cellcolor{LightSkyBlue1}\ldots & \cellcolor{LightSkyBlue1}\ldots & \cellcolor{LightSkyBlue1}-31.3 &  &  &  &  &  &  &  &  &  \\ 
    \textbf{29} & \cellcolor{LightSkyBlue1}-30 & \cellcolor{LightSkyBlue1}-30.1 & \cellcolor{LightSkyBlue1}-30.2 & \cellcolor{LightSkyBlue1}-30.3 & \cellcolor{LightSkyBlue1}\ldots & \cellcolor{LightSkyBlue1}\ldots & \cellcolor{LightSkyBlue1}-32.3 & \cellcolor{LightSkyBlue1}-33.1 &  &  &  &  &  &  &  &  \\ 
    \textbf{30} & \cellcolor{LightSkyBlue1}-31 & \cellcolor{LightSkyBlue1}-31.1 & \cellcolor{LightSkyBlue1}-31.2 & \cellcolor{LightSkyBlue1}-31.3 & \cellcolor{LightSkyBlue1}\ldots & \cellcolor{LightSkyBlue1}\ldots & \cellcolor{LightSkyBlue1}-33.3 & \cellcolor{LightSkyBlue1}-34.1 & \cellcolor{DarkSeaGreen2}-33.6 &  &  &  &  &  &  &  \\ 
    \textbf{31} & \cellcolor{LightSkyBlue1}-32 & \cellcolor{LightSkyBlue1}-32.1 & \cellcolor{LightSkyBlue1}-32.2 & \cellcolor{LightSkyBlue1}-32.3 & \cellcolor{LightSkyBlue1}\ldots & \cellcolor{LightSkyBlue1}\ldots & \cellcolor{LightSkyBlue1}-34.3 & \cellcolor{LightSkyBlue1}-35.1 & \cellcolor{DarkSeaGreen2}-34.6 & \cellcolor{DarkSeaGreen2}-34.2 &  &  &  &  &  &  \\ 
    \textbf{32} & \cellcolor{LightSkyBlue1}-33 & \cellcolor{LightSkyBlue1}-33.1 & \cellcolor{LightSkyBlue1}-33.2 & \cellcolor{LightSkyBlue1}-33.3 & \cellcolor{LightSkyBlue1}\ldots & \cellcolor{LightSkyBlue1}\ldots & \cellcolor{LightSkyBlue1}-35.3 & \cellcolor{LightSkyBlue1}-36.1 & \cellcolor{DarkSeaGreen2}-35.6 & \cellcolor{DarkSeaGreen2}-35.2 & \cellcolor{DarkSeaGreen2}-35.0 &  &  &  &  &  \\ 
    \textbf{33} & \cellcolor{LightSkyBlue1}-34 & \cellcolor{LightSkyBlue1}-34.1 & \cellcolor{LightSkyBlue1}-34.2 & \cellcolor{LightSkyBlue1}-34.3 & \cellcolor{LightSkyBlue1}\ldots & \cellcolor{LightSkyBlue1}\ldots & \cellcolor{LightSkyBlue1}-36.3 & \cellcolor{LightSkyBlue1}-37.1 & \cellcolor{DarkSeaGreen2}-36.6 & \cellcolor{DarkSeaGreen2}-36.2 & \cellcolor{DarkSeaGreen2}-36.0 & \cellcolor{DarkSeaGreen2}-35.9 &  &  &  &  \\ 
    \textbf{34} & \cellcolor{LightSkyBlue1}-35 & \cellcolor{LightSkyBlue1}-35.1 & \cellcolor{LightSkyBlue1}-35.2 & \cellcolor{LightSkyBlue1}-35.3 & \cellcolor{LightSkyBlue1}\ldots & \cellcolor{LightSkyBlue1}\ldots & \cellcolor{LightSkyBlue1}-37.3 & \cellcolor{LightSkyBlue1}-38.1 & \cellcolor{DarkSeaGreen2}-37.6 & \cellcolor{DarkSeaGreen2}-37.2 & \cellcolor{DarkSeaGreen2}-37.0 & \cellcolor{DarkSeaGreen2}-36.9 & \cellcolor{DarkSeaGreen2}-36.8 &  &  &  \\ 
    \textbf{35} & \cellcolor{LightSkyBlue1}-36 & \cellcolor{LightSkyBlue1}-36.1 & \cellcolor{LightSkyBlue1}-36.2 & \cellcolor{LightSkyBlue1}-36.3 & \cellcolor{LightSkyBlue1}\ldots & \cellcolor{LightSkyBlue1}\ldots & \cellcolor{LightSkyBlue1}-38.3 & \cellcolor{LightSkyBlue1}-39.1 & \cellcolor{DarkSeaGreen2}-38.6 & \cellcolor{DarkSeaGreen2}-38.2 & \cellcolor{DarkSeaGreen2}-38.0 & \cellcolor{DarkSeaGreen2}-37.9 & \cellcolor{DarkSeaGreen2}-37.8 & \cellcolor{DarkSeaGreen2}-37.8 &  &  \\ 
    \textbf{36} & \cellcolor{LightSkyBlue1}-37 & \cellcolor{LightSkyBlue1}-37.1 & \cellcolor{LightSkyBlue1}-37.2 & \cellcolor{LightSkyBlue1}-37.3 & \cellcolor{LightSkyBlue1}\ldots & \cellcolor{LightSkyBlue1}\ldots & \cellcolor{LightSkyBlue1}-39.3 & \cellcolor{LightSkyBlue1}-40.1 & \cellcolor{DarkSeaGreen2}-39.6 & \cellcolor{DarkSeaGreen2}-39.2 & \cellcolor{DarkSeaGreen2}-39.0 & \cellcolor{DarkSeaGreen2}-38.9 & \cellcolor{DarkSeaGreen2}-38.8 & \cellcolor{DarkSeaGreen2}-38.8 & \cellcolor{DarkSeaGreen2}-38.7 &  \\ 
    \textbf{37} & \cellcolor{LightSkyBlue1}-38 & \cellcolor{LightSkyBlue1}-38.1 & \cellcolor{LightSkyBlue1}-38.2 & \cellcolor{LightSkyBlue1}-38.3 & \cellcolor{LightSkyBlue1}\ldots & \cellcolor{LightSkyBlue1}\ldots & \cellcolor{LightSkyBlue1}-40.3 & \cellcolor{LightSkyBlue1}-41.1 & \cellcolor{DarkSeaGreen2}-40.6 & \cellcolor{DarkSeaGreen2}-40.2 & \cellcolor{DarkSeaGreen2}-40.0 & \cellcolor{DarkSeaGreen2}-39.9 & \cellcolor{DarkSeaGreen2}-39.8 & \cellcolor{DarkSeaGreen2}-39.8 & \cellcolor{DarkSeaGreen2}-39.7 & \cellcolor{DarkSeaGreen2}-39.7 \\ 
\end{tabular}
\end{adjustbox}
\end{table}

\section{The Case $w>k$}
\label{s:bigw}

A straightforward lower bound for $\R_\sigma(k,w)$ (and for the expected density of \emph{any} particular order) is $\sigma^{-k}$: every position of the minimal $k$-mer is marked, and the expected density of such positions in a random string is $\sigma^{-k}$.
If $w$ is very big compared to $k$, then almost all windows contain the minimal  $k$-mer; thus,  $\R_\sigma(k,w)$ approaches $\sigma^{-k}$.
The next proposition clarifies what is ``very big'' in this context.

\begin{proposition}\label{p:bigw}
Let $N=\sigma^k$ and let $w=\frac{\sigma}{\sigma-1}N(\ln N+g(N))$ for arbitrary fixed positive function $g$.
Then $\R_\sigma(k,w)=(1+O(e^{-g(N)}))\sigma^{-k}$.
\end{proposition}
\begin{proof}
    Let $A_{\sigma,u}(n)$ be the number of $\sigma$-ary $n$-strings having no occurrence of the $k$-mer $u$ and let $A_{\sigma,k}(n)=\max\{A_{\sigma,u}(n)\mid u \text{ is a } k\text{-mer}\}$.
    Since some $k$-mer should be marked in each window having no occurrence of the minimal $k$-mer, we get the upper bound $\R_\sigma(k,w)\le \frac{1}{\sigma^k}+\frac{A_{\sigma,k}(w+k-1)}{\sigma^{w+k-1}}$.
    The function $A_{\sigma,k}(n)$ can be estimated by the method of Guibas and Odlyzko \cite{GuOd78, GuOd81}; we use the bound based on \cite[Sect.~4]{RuSh16}: $A_{\sigma,k}(n)\le (1+\frac{k}{\sigma^k})(\sigma-\frac{\sigma-1}{\sigma^k})^n$.
    Then we have
    \begin{equation*}
        \tfrac{A_{\sigma,k}(w+k-1)}{\sigma^{w+k-1}} \le (1+\tfrac{k}{\sigma^k})(1-\tfrac{\sigma-1}{\sigma^{k+1}})^{w+k-1} < 
        (1+\tfrac{k}{\sigma^k})e^{-\frac{\sigma-1}{\sigma^{k+1}}(w+k-1)}\,.
    \end{equation*}
    Since $(1{+}\tfrac{k}{\sigma^k})e^{-\frac{(\sigma-1)(k-1)}{\sigma^{k+1}}}\!\!=O(1)$, by substituting $w=\frac{\sigma}{\sigma-1}N(\ln N{+}g(N))$ we get
    $$
    \hspace*{5mm} \R_\sigma(k,w)\le \tfrac{1}{\sigma^k}+ O(e^{-\frac{\sigma-1}{\sigma^{k+1}}\cdot\frac{\sigma^{k+1}}{\sigma-1}(k\ln\sigma{+}g(N))}) = \tfrac{1}{\sigma^k}+ O(\tfrac{e^{-g(N)}}{\sigma^k}).\hspace*{5mm} \qed
    $$
\end{proof}

\section{Discussion and Future Work}
\label{s:disc}

Random minimizer is an object interesting for both theory and practice.
Its main characteristic is the density $\R_\sigma(k,w)$ studied in this paper.
We provide a detailed description of density for the case $w\le k$; the only remaining point of interest is to prove that for every $\sigma$ the density passes the limit value $\frac{2}{w+1}$ only once (see Remark~\ref{r:transit} and Tables~\ref{tbl:binary},~\ref{tbl:10ary}).

The case $w>k$ presents more open problems, which can can be easily seen if we plot $\R_\sigma(k,w)$ as a function of $w$ for some $\sigma$ and $k$ (see Figure~\ref{fig:s}).
We know approximate values of this function for small $w$ (see \eqref{e:dev} and Theorem~\ref{t:horiz}) and for very big $w$ (Proposition~\ref{p:bigw}), but to fill the intermediate range is an open problem.
The following simple lemma shows that $\R_\sigma(k,w)$ is monotone in $w$.

\begin{lemma}\label{l:Rmonotone}
    $\R_\sigma(k,w+1)\le\R_\sigma(k,w)$.
\end{lemma}
\begin{proof}
    Consider an arbitrary minimizer $\S_\rho$ with parameters $k,w+1$ and an arbitrary window $u=u[1..w{+}k]$.
    When processing $u$, the scheme chooses some $k$-mer $u[i..i{+}k{-}1]$.
    Then $\S_\rho$ with parameters $k,w$ chooses the same $k$-mer $u[i..i{+}k{-}1]$ when processing the window $u[1..w{+}k{-}1]$, or $u[2..w{+}k]$, or both.
    Hence for any processed string, the set of positions marked by $\S_\rho$ with the parameters $k,w+1$ forms a subset of positions marked by $\S_\rho$ with the parameters $k,w$.
    The lemma now follows from definitions.
    \qed
\end{proof}

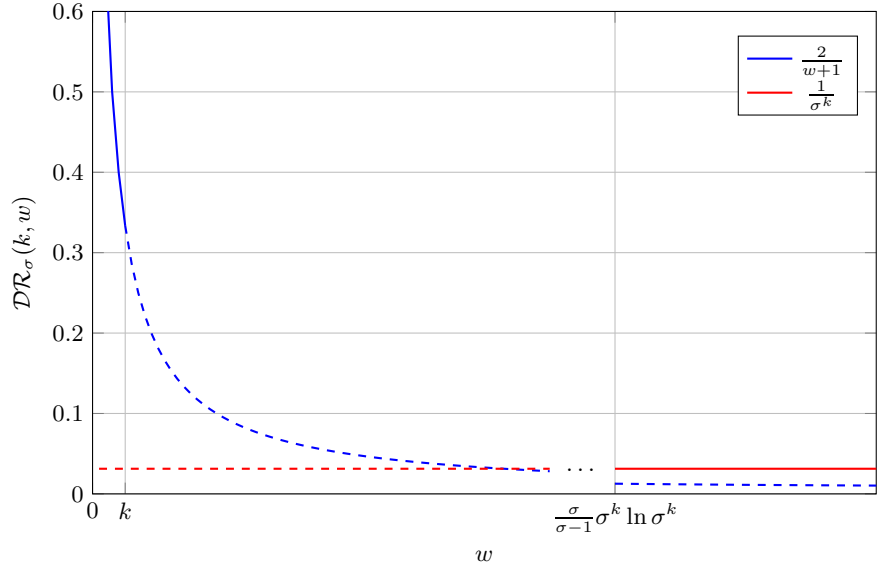
\begin{figure}[!ht]
    \centering



    \begin{tikzpicture}
        \begin{axis}[
            width=12cm, height=8cm, 
            xlabel={$w$}, ylabel={$\R_\sigma(k, w)$},
            title={Plot of approximations for $\R_\sigma(k, w)$ ($\sigma=2$, $k=5$)},
            xmin=0, xmax=120, ymin=0, ymax=0.6, 
            xtick={0, 5, 80}, 
            xticklabels={$0$, $k$, $\tfrac{\sigma}{\sigma-1}\sigma^k \ln\sigma^k$},
            grid=major,
            legend style={font=\small, at={(0.9,0.95)}, anchor=north, legend columns=1}, 
            ]

            \addplot[blue, thick] table {
                1 1
                2 0.666667
                3 0.5
                4 0.4
                5 0.333333
            };
            \addlegendentry{$\frac{2}{w+1}$}
            \addplot[red, thick] table {
                80 0.03125
                90 0.03125
                100 0.03125
                110 0.03125
                120 0.03125
            };
            \addlegendentry{$\frac{1}{\sigma^k}$}
            \addplot[blue, thick, dashed] table {
                5 0.33333
                6 0.285714
                7 0.25
                8 0.222222
                9 0.2
                10 0.181818
                11 0.166667
                12 0.153846
                13 0.142857
                14 0.133333
                15 0.125
                16 0.117647
                17 0.111111
                18 0.105263
                19 0.1
                20 0.0952381
                21 0.0909091
                22 0.0869565
                23 0.0833333
                24 0.08
                25 0.0769231
                26 0.0740741
                27 0.0714286
                28 0.0689655
                29 0.0666667
                30 0.0645161
                31 0.0625
                32 0.0606061
                33 0.0588235
                34 0.0571429
                35 0.0555556
                36 0.0540541
                37 0.0526316
                38 0.0512821
                39 0.05
                40 0.0487805
                41 0.047619
                42 0.0465116
                43 0.0454545
                44 0.0444444
                45 0.0434783
                46 0.0425532
                47 0.0416667
                48 0.0408163
                49 0.04
                50 0.0392157
                51 0.0384615
                52 0.0377358
                53 0.037037
                54 0.0363636
                55 0.0357143
                56 0.0350877
                57 0.0344828
                58 0.0338983
                59 0.0333333
                60 0.0327869
                61 0.0322581
                62 0.031746
                63 0.03125
                64 0.0307692
                65 0.030303
                66 0.0298507
                67 0.0294118
                68 0.0289855
                69 0.0285714
                70 0.028169
            };

            \addplot[red, thick, dashed] table {
                1 0.03125
                10 0.03125
                20 0.03125
                30 0.03125
                40 0.03125
                50 0.03125
                60 0.03125
                70 0.03125
            };

            \addplot[blue, thick, dashed] table {
                80 0.0126582
                81 0.0125786
                82 0.0125
                83 0.0124224
                84 0.0123457
                85 0.0122699
                86 0.0121951
                87 0.0121212
                88 0.0120482
                89 0.011976
                90 0.0119048
                91 0.0118343
                92 0.0117647
                93 0.0116959
                94 0.0116279
                95 0.0115607
                96 0.0114943
                97 0.0114286
                98 0.0113636
                99 0.0112994
                100 0.01123596
                110 0.01075269
                120 0.01030928
            };

            \node at (axis cs:75,0.03) {\dots};
        \end{axis}
    \end{tikzpicture}

    \caption{For $w\le k$ we have $\R_\sigma(k,w)$ very close to $\tfrac{2}{w+1}$ (by \eqref{e:R} and Theorem~\ref{t:horiz}).
    For $w>\tfrac{\sigma}{\sigma-1}\sigma^k\ln\sigma^k$ we have $\R_\sigma(k,w)$ very close to $\tfrac{1}{\sigma^k}$ (by Proposition~\ref{p:bigw}).
    There is still a gap in our knowledge in the range $k<w<\tfrac{\sigma}{\sigma-1}\sigma^k\ln\sigma^k$, where we just know that $\R_\sigma(k,w)$ is monotone decreasing (by Lemma~\ref{l:Rmonotone}).}
    \label{fig:s}
\vspace*{-3mm}
\end{figure}

In order to describe the behaviour of $\R_\sigma(k,w)$ for ``medium'' values of $w$, it is necessary to describe the ranges of $w=w(\sigma,k)$ where the density is $\frac{2+o(1)}{w}$; $\le \frac{C}{w}$ for an absolute constant $C>2$; $\le \frac{C}{\sigma^k}$ for an absolute constant $C>1$.
From \cite{ZKM20} we have the lower bound for the first range: all values $w=O(\sigma^{k/3-\varepsilon})$ are inside it.
We believe that this bound is a big underestimate.
Another result of \cite{ZKM20} implies an upper bound $w=O(\sigma^k)$ for the second range, for every $C$. 
Our conjecture for the third range is $w=\Omega(\sigma^k\cdot k\ln\sigma)$.

\subsubsection*{Acknowledgments.}
S. Golan is supported by Israel Science Foundation grant no. 810/21.
A. Shur is supported by the ERC grant MPM no. 683064 under the EU’s
Horizon 2020 Research and Innovation Programme and by the State of Israel through the Center for Absorption in Science of the Ministry of Aliyah and Immigration.

\bibliographystyle{splncs04}
\bibliography{bib}

\end{document}